\documentclass[12pt, a4paper, draft]{article}

\usepackage{amsmath, amsfonts, amsthm, amssymb}

\newcommand     {\C}            {{\mathbb C}}
\newcommand     {\F}            {{\mathbb F}}
\newcommand     {\Hom}          {\sym{Hom}}
\newcommand     {\HP}           {{\mathbb H}}
\newcommand     {\Ind}          {\sym{Ind}}
\newcommand     {\Inv}          {\sym{Inv}}
\newcommand     {\JG}[2][n]     {\sym{J}_{#1}({#2})}
\newcommand     {\Mp}           {\mathbb{M}}
\newcommand     {\Mpp}[1][\Z]   {\sym{Mp}(2,{#1})}
\newcommand     {\Q}            {{\mathbb Q}}

\newcommand     {\Ring}[1]      {\Z/{#1}}
\newcommand     {\SL}[1][\Z]    {\sym{SL}(2,{#1})}
\newcommand     {\Z}            {{\mathbb Z}}

\newcommand     {\e}            {\sym{e}}
\newcommand     {\conj}[1]      {{#1}^c}
\newcommand     {\down}[1]      {{#1}'}
\newcommand     {\pt}[1]        {\sym{e}_{#1}}
\newcommand     {\sym}[1]       {\operatorname{#1}}
\newcommand     {\up}[1]        {\widetilde{#1}}

\theoremstyle{plain}
\newtheorem{Theorem}{Theorem}
\newtheorem*{Corollary}{Corollary}
\newtheorem{Lemma}{Lemma}
\theoremstyle{definition}
\newtheorem*{Definition}{Definition}
\newtheorem*{Remark}{Remark}

\title{%
  Jacobi Forms of Critical Weight and Weil Representations
}

\author{%
  Nils-Peter Skoruppa}

\date{}

\begin{document}

\maketitle


\begin{abstract}
\noindent
Jacobi forms can be considered as vector valued modular forms, and
Jacobi forms of critical weight correspond to vector valued modular
forms of weight $\frac12$. Since the only modular forms of weight
$\frac12$ on congruence subgroups of $\SL$ are theta series the theory
of Jacobi forms of critical weight is intimately related to the theory
of Weil representations of finite quadratic modules. This article
explains this relation in detail, gives an account of various facts
about Weil representations which are useful in this context, and it
gives some applications of the theory developed herein by proving
various vanishing theorems and by proving a conjecture on Jacobi forms
of weight one on $\SL$ with character.
(2000 Mathematics Subject Classification:
11F03 
11F50 
11F27 
)
\end{abstract}

\tableofcontents

\section{Introduction}

All Siegel modular forms of degree $n$ and weight $k<\frac n2$ are
singular. Similarly, every orthogonal modular form associated to a
quadratic module of signature~$(2,n+1)$ and of weight~$k<\frac n2$ is
singular. The weights in the given ranges are called
singular weights, a terminology, which was introduced by Resnikoff.
Spaces of modular forms of singular weight are well-understood
(cf.~\cite{Fr 91} and the literature cited therein for singular Siegel
modular forms, and~\cite{Re 75} for singular orthogonal modular
forms).  For critical weights much less is known.  Here, by
{\it critical weight}, we understand the weight $\frac n2$ in the theory of
Siegel modular forms of degree $n$ and in the theory of orthogonal
modular forms of signature $(2,n+1)$ respectively. \cite{We 92} gives some results
for Siegel modular forms of degree 2 and critical weight~$1$.  More
recently, it was shown in \cite{I-S 06} that there are no cusp forms
of degree 2 and weight 1 (with trivial character) on the subgroups~$\Gamma_0(l)$ of
$\sym{Sp}(2,\Z)$.

In a joint project Ibukiyama and the author~\cite{I-S 07} take up a
more systematic study of Siegel and orthogonal modular forms of
singular weight, a first result being~\cite{I-S 06}. One of the main
tools used in this project is the Fourier-Jacobi expansion of the
modular forms in question, more precisely, the Fourier-Jacobi
expansion with respect to Jacobi forms of degree~$1$. For modular
forms of critical weight the Fourier-Jacobi coefficients will then be
Jacobi forms whose index is a symmetric positive definite matrix of
size $n$ and whose weight equals $\frac {n+1}2$. We shall call this
weight {\it critical} for Jacobi forms of degree~$1$ whose matrix
index is of size $n$.

Jacobi forms of critical weight play also a role in recent work of
Gritsenko, Sankaran and Hulek on the the geometry of moduli spaces of
K3-surfaces. The explicit construction of Jacobi forms of small weight
is an essential tool in their studies. In particular, they use Jacobi
forms with scalar index of weight 1 (hence of critical weight) and
with a nice product expansion to construct those Jacobi forms which
they need. These Jacobi forms of weight 1 with product expansion are
special instances of socalled {\it thetablocks} (cf.~the discussion before
Theorem~\ref{thm:weight-one-vanishing} in section~\ref{sec:critical-weight} for a
more precise statement). It is an interesting question whether all
Jacobi forms of small weight (and of scalar index) can be obtained by thetablocks
as it is suggested by numerical experiments. For
Jacobi forms of weight $1/2$ this question can be
answered affirmatively using the description of these forms given in
Corollary to Theorem~\ref{thm:first-isomorphism} below
(cf.~\cite{GSZ 07} for details).  For Jacobi forms
of weight 1 (and of scalar index) we are similarly lead to the problem
of finding a sufficiently explicit description of these forms. Such a
description will drop out as part of the considerations of this
article. In fact, we shall show that all Jacobi forms of weight 1
with character $\varepsilon^8$ (see section~\ref{sec:Notations}) are linear combinations of
thetablocks.  A more systematic theory of thetablocks will be developed in
a joint article of V.~Gritsenko, D.~Zagier and the author \cite{GSZ 07}.
 
In view of the aformentioned applications it seems to be worthwhile
to develop a systematic theory of Jacobi forms of critical
weight. These Jacobi forms can be studied as vector valued elliptic
modular forms of (critical) weight~$\frac 12$. The first main tool
for setting up such a theory is
an explicit description of this connection. The second main tool
is then the fact that elliptic modular forms of weight~$\frac 12$ are
theta series~\cite{Se-S 77}, and the third main tool is the
decomposition of the space of modular forms of weight~$\frac 12$ with
respect to the action of the metaplectic double cover of
$\SL$~\cite{Sko 85}. The latter makes it possible to describe spaces of
vector valued elliptic modular forms of weight~$\frac 12$ in
sufficiently explicit form.  A deeper study of these explicit
descriptions requires a deeper understanding of Weil representations
of $\SL$ (or rather its metaplectic double cover) associated to finite
quadratic modules. In fact, it turns out that spaces of Jacobi forms
of critical weight are always naturally isomorphic to spaces of invariants
of suitable Weil representations
(cf.~Theorem~\ref{thm:third-isomorphism} and the subsequent remark).
 
The present article aims to pave the way for a theory of Jacobi forms of critical weight
and, in particular, for the aforementioned projects
\cite{I-S 07}, \cite{GSZ 07} and possible other applications.
We propose a formal framework for such a theory, we
give an account of what one can prove for Jacobi forms of critical weight within this framework,
and we describe the necessary tools from the theory of Weil representations and
the theory of modular forms of weight $\frac12$ which are needed.
Some of the following results are new,
others are older but are not easily available elsewhere
or are even unpublished.

The plan of this article is as follows: In
section~\ref{sec:weil-representations} we discuss Weil representations
of finite quadratic modules. All the considerations of this section
are mainly motivated by the question for the
invariants of a given Weil representation. 
In section~\ref{sec:connection-to-vector-valued-modular-forms}
the relation between (vector valued)
Jacobi forms and vector valued modular forms is made explicit.  As an
application we obtain (cf.~Theorem~\ref{thm:dimension-formula})
a dimension formula for spaces of Jacobi forms
(of arbitrary matrix index).
In section~\ref{sec:critical-weight} we finally turn to Jacobi forms
of critical weight.  We shall see that Jacobi forms of critical weight
are basically invariants of certain Weil representations associated to
finite quadratic modules (cf.~Theorem~\ref{thm:third-isomorphism} and
the subsequent remarks). We apply our theory to prove various
vanishing results in Theorems~\ref{thm:first-vanishing-result},
~\ref{thm:weight-one-vanishing}, Corollary to
Theorem~\ref{thm:vanishing-on-Gamma_0}, and to prove in
Theorem~\ref{thm:quarks} that Jacobi forms of weight 1 and character
$\varepsilon^8$ are obtained by theta blocks.  In
section~\ref{sec:proofs} we append those proofs which have been
omitted in the foregoing sections because of their more technical or
computational nature.  For the convenience of the reader we insert a
section~\ref{sec:Notations} which contains a glossary of the main
notations.

\section{Notation}
\label{sec:Notations}

We use $[a,b;c,d]$ for the $2\times2$ matrix with first and second row
equal to $(a,b)$ and $(c,d)$, respectively.  If $F$ is an $n\times n$
matrix and $x$ a column vector of size $n$, we write $F[x]$ for
$x^tFx$. We use $\e(X)$ for $\exp(2\pi i X)$, and $\e_m(X)$ for
$\exp(2\pi i X/m)$.  For integers $a$ and $b$, the notation
$a|b^\infty$ indicates that every prime divisor of $a$ is also a
divisor of $b$. For relatively prime $a$ and $b$ we use~$\left(\frac
ab\right)$ for the usual generalized Jacobi-Legendre Symbol with the
(additional) convention $\left(\frac a{2^\beta}\right)=-1$ if $a\equiv
\pm3\bmod8$ and $\beta$ is odd and  $\left(\frac a{2^\beta}\right)=+1$
if $a\equiv\pm1\bmod8$ or $\beta$ is even.
We summarize (in roughly alphabetical order) the most
important notations of this article:
\begin{list}{-}{\labelwidth=50pt}
\item[$\C(\chi)$]
For a character $\chi$ of the metaplectic cover $\Mp=\Mpp$, the
$\Mp$-module with underlying vector space $\C$ and with $\Mp$-action
$(g,z) \mapsto \chi(z)g$.
\item[$\varepsilon$]
The one dimensional character of $\Mp=\Mpp$ given by
$\varepsilon(A,w)=\eta(A\tau)/\big(w(\tau)\eta(\tau)\big)$, where
$\eta$ is the Dedekind eta function.
\item[$\vartheta(\tau,z)$]
The Jacobi form $q^{1/8}\big(\zeta^{1/2}-\zeta^{-1/2}\big)
\prod_{n\ge 1}
\big(1-q^n\big)
\big(1-q^n\zeta\big)
\big(1-q^n\zeta^{-1}\big)$.
\item[$\up\Gamma$]
For a subgroup $\Gamma$ of $\SL$ the inverse image
of $\Gamma$ in $\Mpp$ under the natural projection onto the first
factor.
\item[$\down\Gamma$]
For a subgroup $\Gamma$ of $\Mpp$ its image under the
natural projection onto the first factor.
\item[$\Gamma(4m)^*$]
The (normal) subgroup of $\Mpp$ consisting of all pairs
$\big(A,j(A,\tau)\big)$, where $A$ is in the principal congruence
subgroup $\Gamma(4m)$ of matrices which are $1$ modulo $4m$, and where
$j(A,\tau)$ stands for the standard multiplier system from the theory
of modular forms of half-integral weight,
i.e.~$j(A,\tau)=\theta(A\tau)/\theta(\tau)$ where
$\theta(\tau)=\sum_{r\in\Z}\exp(2\pi i \tau r^2)$.
\item[$\HP$]
The upper half plane of complex numbers.
\item[$\Inv(V)$]
The space of $\Mpp$-invariant elements in an $\Mpp$-module $V$.
\item[$M_k(\Gamma)$]
For a subgroup $\Gamma$ of $\SL$, the space of modular forms of weight
$k$ on $\Gamma$; if $k\in\frac12+\Z$ then it is assumed that $\Gamma$
is contained in $\Gamma_0(4)$, and every $f$ in $M_k(\Gamma)$ satisfies
$f(A\tau)j(A,\tau)^{-2k}=f(\tau)$ for $A\in\Gamma$ (cf.~above for
$j(A,\tau)$).
\item[$\Mp$]
Abbreviation for $\Mpp$, see below.
\item[$\Mpp$]
The metaplectic double cover of $\SL$, i.e.~the group of all pairs
$(A,w)$, where $A\in\SL$ and $w$ is a holomorphic function on $\HP$
such that $w(\tau)^2=c\tau+d$, equipped with the composition law
$(A,w)\cdot(B,v)=\big(AB,w(B\tau)v(\tau)\big)$.
\item[$S$]
The matrix $[0,-1;1,0]$.
\item[$T$]
The matrix $[1,1;0,1]$.
\item[$\vartheta_{F,x}$]
For a symmetric, half-integral positive definite matrix $F$
of size $n$ and a column vector $x\in\Z^n$,
$$
	 \vartheta_{F,x}(\tau,z)
	 =
	 \sum_{\begin{subarray}{c}r\in\Z^n\\
		  r\equiv x \bmod 2F\Z^n\end{subarray}}
	 \e\big(\tau \frac14 F^{-1}[r] + r^tz\big)
	 \qquad
	 (\tau\in \HP,\; z\in\C^n)
$$
\item[$V^*$]
For a left $G$-module $V$, the right $G$-module with the dual of the
complex vector space $V$ as underlying space equipped with the
$G$-action $(\lambda,g)\mapsto \big(v\mapsto\lambda(gv)\big)$.
\item[$\conj V$]
For a left $G$-module $V$, the left $G$-module with the dual of the
complex vector space $V$ as underlying space equipped with the
$G$-action $(g,\lambda)\mapsto \big(v\mapsto\lambda(g^{-1}v)\big)$.
\item[$w_A$]
For a matrix $A\in\SL$, the function $w_A(\tau)=\sqrt{a\tau+b}$, where
the square root is chosen in the right half plane or on the
nonnegative imaginary axes.
\end{list}

\section{Weil Representations of $\Mpp$}
\label{sec:weil-representations}

In this section we recall those basic facts about Weil representations
of $\Mp=\Mpp$ associated to finite quadratic modules which we shall need in
the sequel. These representations have been studied by \cite{Kl 46},
\cite{N-W 76}, \cite{Ta 67} et al..  A more complete account of this
theory as well as some deeper facts which can not be found in the
literature will be given in~\cite{Sko 07}.

By a finite quadratic module $M$ we understand a finite abelian group
$M$ endowed with a quadratic form $Q_M:M\rightarrow \Q/\Z$. Thus, by
definition, we have $Q_M(ax)=a^2Q_M(x)$ for all $x$ in $M$ and all integers
$a$, and the application $B_M(x,y):=Q_M(x+y)-Q_M(x)-Q_M(y)$ defines a
$\Z$-bilinear map $B_M:M\times M\rightarrow \Q/\Z$. All quadratic
modules occurring in the sequel will be assumed to be
non-degenerate if not otherwise stated. Recall that $M$ is called non-degenerate if
$B_M(x,y)=0$ for all $y$ is only possible for~$x=0$.

Denote by $\C[M]$ the complex vector space of all formal linear
combinations $\sum_x \lambda(x)\,\pt x$, where $\pt x$, for $x\in M$,
is a symbol, where $\lambda(x)$ is a complex number and where the sum
is over all $x$ in $M$.  We define an action of $(T,w_T)$ and
$(S,w_S)$ on $\C[M]$ by
\begin{equation*}
\label{eq:action}
\begin{split}
	 (T,w_T)\,\pt x &= e(Q_M(x))\,\pt x\\
	 (S,w_S)\,\pt x &= \sigma\,|M|^{-\frac12}\sum_{y\in M} \pt y\, e(-B_M(y,x))
	 ,
\end{split}
\end{equation*}
where
$$
	 \sigma=\sigma(M)=|M|^{-\frac12}\sum_{x\in M} e(-Q_M(x))
	 .
$$
This can be extended to an action of the metaplectic group $\Mp$
\cite{Sko 07}, and we shall use~$W(M)$ to denote the $\Mp$-module with
underlying space~$\C[M]$. This action factors through $\SL$ if and
only if $\sigma^4=1$ \cite{Sko 07}; in general, $\sigma$ is an eighth
root of unity. That these formulas define an action of $\SL$ if $\sigma^4=1$
is well-known (cf.~e.g.~\cite{N 76}).

It follows immediately from the defining formulas for the Weil
representations that the pairing
$\{-,-\}:W(M)\otimes W(-M)\rightarrow \C$ given by $\{\pt x,\pt y\}=1$
if $x=y$ and $\{\pt x,\pt y\}=0$
otherwise is invariant under $\Mp$. Here $-M$ denotes the quadratic
module with the same underlying group as $M$ but with the quadratic
form $x\mapsto -Q_M(x)$. The $\Mp$-invariance of this pairing is just
another way to state that the matrix representation of $\Mp$ afforded
by $W(M)$ with respect to the basis $\pt x$ ($x\in M$) is unitary. The
perfect pairing induces a natural isomorphism     of $W(M)^c$ and
$W(-M)$.

A standard example for a quadratic module is the determinant
group~$D_F$ of a symmetric non-degenerate half-integral matrix $F$.
By half-integral we mean that $2F$ has integer entries and even
integers on the diagonal.  The quadratic module has $D_F=\Z^n/2F\Z^n$
as underlying abelian group.  The quadratic form on $D_F$ is the one
induced by the quadratic form $x\mapsto \frac 14 F^{-1}[x]$ on
$\Z$. We shall henceforth write $W(F)$ for $W(D_F)$. Special instances
are the quadratic modules $D_m=(\Ring{2m}, x\mapsto \frac{x^2}{4m})$ for
integers $m\not=0$ and their associated Weil representations
$W(m)$.  The decomposition of $W(m)$ into irreducible $\Mp$-modules
was given in~\cite[Theorem 1.8, p.22]{Sko 85}.  We shall recall this
in section~\ref{sec:proofs}.

The {\it level} $l$ of a quadratic module is the smallest positive
integer $l$ such that $lM=0$ and $lQ_M=0$. The level of $D_F$
coincides with the level of $2F$ as defined in the theory of quadratic
forms, i.e.~it coincides with the smallest integer $l>0$ such that
$lF^{-1}/2$ is half-integral.  If $l$ denotes the level of $M$ and if
$\sigma(M)^4=1$, i.e.~if $W(F)$ can be viewed as $\SL$-module, then the
group $\Gamma(l)$ acts trivially on $W(M)$; if $\sigma(M)^4\not=1$
then $l$ is divisible by $4$ and $\Gamma(l)^*$ acts trivially on
$W(M)$ \cite{Sko 07}.

By $O(M)$ we denote the orthogonal group of a quadratic module $M$,
i.e.~the group of all automorphisms of the underlying abelian group of
$M$ such that $Q_M\circ\alpha=Q_M$. The group $O(m)$, for an integer
$m>0$, is the group of left multiplications of $\Ring{2m}$ by elements
$a$ in $(\Ring{2m})^*$ satisfying $a^2=1$. Its order equals the number of
prime factors of $m$.  The group $O(M)$ acts on $\C[M]$ in the obvious
way. It is easily verified from the defining equations for the action
of $(S,w_S)$ and $(T,w_T)$ on $\C[M]$ that the action of $O(M)$
intertwines with the action of $\Mp$ on $W(M)$. In particular, if~$H$
is a subgroup of $O(M)$ then the subspace $W(M)^H$ of elements in
$W(F)$ which are invariant by $H$ is a $\Mp$-submodule of $W(M)$.

It will turn out that spaces of Jacobi forms of critical weight are
intimately related to the spaces of invariants of Weil
representations. For a Weil representation $W(M)$ we use $\Inv(M)$ for
the subspace of elements in $W(M)$ which are invariant under the action of $\Mp$.
For a mtrix $F$, we also write $\Inv(F)$ for $\Inv(D_F)$.
If $\sigma(M)^4\not=1$ then the action of $\Mp$ on $W(M)$ does not factor
through $\SL$, hence $(1,-1)$ does not act trivially on $W(M)$. It can
be checked (or cf.~\cite{Sko 07}) that $(1,-1)$ acts as nontrivial homotethy,
whence $\Inv(M)=0$.

If $\sigma(M)$ is a fourth root of unity then the question for
$\Inv(M)$ is much more subtle.  Roughly speaking there will be
invariants if $M$ is big enough, and the spaces $\Inv(M)$ fall into
several natural categories according to certain local invariants of
quadratic modules \cite{Sko 07}.

There is one obvious way to construct invariants. Namely, suppose that
$M$ contains an isotropic self-dual subgroup $U$, i.e~a subgroup $U$
such that $Q_M(x)=0$ for all $x$ in $U$ and such that the {\it dual}
$U^*$ of $U$ equals $U$ (where, for a submodule $U$, the dual $U^*$ is,
by definition, the submodule of all $y$ in $M$ satisfying $B_M(x,y)=0$
for all $x$ in $U$). Then the element $I_U:=\sum_{x\in U}\pt x$ is
invariant under $\SL$ (as follows immediately from the defining
equations for the action of $S$ and $T$ and the fact that these
matrices generate $\SL$). Note that here $|M|$ must be a perfect
square (since, for a subgroups $U$ of $M$ one always has an
isomorphism  of abelian groups $M/U\cong U^*$). Also, it is not hard to
check that here $\sigma(M)=1$.

There is one important case where this construction exhausts all
invariants.  We cite some of the results of \cite{Sko 07} which will
clarify this a bit and which will supplement the considerations
in section~\ref{sec:critical-weight}.  For a prime $p$, let $M(p)$ be
the quadratic module with the $p$-part of the abelian group $M$ as
underlying space, equipped with the quadratic form inherited from $M$.

We set
$\sigma_p(M):=\sigma\big(M(p)\big)$.  If $F$ is half-integral and
non-degenerate then
$$
\sigma_p(D_F) =
\begin{cases}
  \e_8\big(p\text{-excess}(2F)\big)&\text{for }p\ge3,\\
  \e_8\big(-\text{oddity}(2F)\big)&\text{for }p=2
\end{cases}.
$$
(For $p$-excess and oddity cf.~\cite[p.~370]{Co-S 88}.) The proof
will be given in~\cite{Sko 07}; we use these formulas in this article
only for the case where $F$ is a scalar matrix, say $F=(n)$. Here these
formulas read
$$
\sigma_p(D_n)=
\begin{cases}
  \sqrt{\left(\frac{-4}q\right)}\left(\frac{-n/q}q\right)&\text{if }p\not=2\\
  e_8(-n/q)\left(\frac{-n/q}{2q}\right)&\text{if }p=2
\end{cases}
,
$$
where $q$ denotes the exact power of $p$ dividing $n$.
These identities
follow directly from the well-known theory of Gauss sums.
\begin{Theorem}
\label{thm:simple-invariants}
Let $M$ be a quadratic module whose order is a perfect square and such
that $\sigma_p(M)=1$ for all primes. Then $\Inv(M)$ is different from
zero. Moreover, $\Inv(M)$ is generated by all $I_U=\sum_{x\in U}\pt x$,
where $U$ runs through the isotropic self-dual subgroups of $M$.
\end{Theorem}

The proof of this theorem is quite tedious. The theorem can be reduced
to a special case of a more general theory concerning invariants of
the Clifford-Weil groups of certain form rings \cite[Theorem
5.5.7]{N-S-R 07}. A more direct proof tailored to the Weil
representations considered here will be given in~\cite{Sko 07}. It is
not hard to show that the assumptions of
Theorem~\ref{thm:simple-invariants} are necessary for $\Inv(M)$ being
generated by invariants of the form $I_U$ (in this article we do not
make use of this).  In general, if $M$ does not satisfy the hypothesis
of Theorem~\ref{thm:simple-invariants}, there might still be
invariants. However, there is one important case, where this is not
the case.

\begin{Theorem}
\label{thm:hauptvermutung-for-small-rank}
Let $M$ be a quadratic module whose order is a power of the
prime $p$. Suppose $\dim_{\F_p} M\otimes \F_p\le 2$. Then
$\Inv(M)\not=0$ if and only if~$|M|$ is a perfect square and
$\sigma(M)=1$
\end{Theorem}
The proof of this theorem will be given
in~\ref{sec:proofs}.
In general, as soon as $\dim_{\F_p} M\otimes \F_p > 2$ the space of
invariants of $D_F$ is nontrivial~\cite{Sko 07}.

If $M$ and $N$ are quadratic modules, we denote by $M\perp N$ the
orthogonal sum of $M$ and $N$, i.e.~the quadratic module whose
underlying abelian group is the direct sum of the abelian groups $M$
and $N$ and whose quadratic form is given by $x\oplus y\mapsto
Q_M(x)+Q_N(y)$.  It is obvious that every $M$ is the orthogonal
sum of its $p$-parts.  From the product formula for quadratic forms
\cite[p.~371]{Co-S 88}
the sum of the numbers $p\text{-excess}(2F)$, taken over all odd $p$,
minus the oddity of $2F$ plus the signature of $2F$ add up to $0$
modulo $8$.  Hence we obtain\footnote{%
In the literature, this formula is sometimes cited as Milgram's formula.
}
$\sigma(D_F)=\e_8(-\text{signature}(2F))$.

A nice functorial (and almost obvious) property is that the
$\Mp$-modules $W(M\perp N)$ and $W(M)\otimes W(N)$ are isomorphic.  In
particular, $W(M)$ is isomorphic to $\bigotimes W\big(M(p)\big)$,
taken over, say, all $p$ dividing the exponent of~$M$.  If $l$ denotes
the level of $M$ then, for each prime $p$, the level of
$M(p)$ equals the $p$-part $l_p$ of $l$,  and then $W(M)$,
$W\big(M(2)\big)$ and $W\big(M(p)\big)$, for odd $p$, factor through
$\Gamma(l)^*$, $\Gamma(l_2)^*$ and $\Gamma(l_p)$, respectively (here
we view $M(p)$ as $\SL$-module which is possible since
$\sigma\big(M(p)\big)^4=1$ as is obvious from the definition of
$\sigma$ in terms of Gauss sums). Since $\Mp$ is isomorphic to the
product of the groups $\Mp/\Gamma(l_2)^*$ and $\SL/\Gamma(l_p)$ we
deduce that in fact $W(M)$, viewed as $\Mp/\Gamma(l)^*$-module, is
naturally isomorphic to the outer tensor product of the
$\Mp/\Gamma(l_2)^*$-module $W\big(M(2)\big)$ and the
$\SL/\Gamma(l_p)$-modules $W\big(M(p)\big)$. In particular, we have a
natural isomorphism
$$
\Inv(M)
\cong
\bigotimes_{p|l} \Inv\big(M(p)\big)
.
$$
The preceding Theorem thus implies
\begin{Theorem} 
\label{thm:small-rank}
Let $F$ be half-integral. Suppose that
$\dim_{\F_p}D_F\otimes F_p\le 2$
for all primes $p$. Then $\Inv(F)\not=0$ if and only if $\det(2F)$
is a perfect square and~$\sigma_p(D_F)=1$ for all primes $p$.
\end{Theorem}
\begin{Remark}
If $F$ is positive-definite, say, of size $n$, then $\sigma(D_F)=\e_8(-n)$.
Thus, if $\dim_{\F_p}D_F\otimes F_p\le 2$ for all primes $p$ then,
by the theorem,
$\Inv(F)=0$ unless $\det(D_F)$ is a perfect square
and $n$ is divisible by $8$.
\end{Remark}
The meaning of the numbers $\sigma_p(M)$ becomes clearer if one
introduces the Witt group of finite quadratic modules
(see~\cite[Ch.~5, \S1]{Sch 84},
or~\cite{Sko 07} for a discussion more adapted to the current situation).
This group generalizes the well-known Witt group of quadratic
spaces, say, over the field $\F_p$, which can be viewed as special
quadratic modules. We call two quadratic modules $M$ and $N$ {\it Witt
equivalent} if they contain isotropic subgroups $U$ and $V$,
respectively, such that $U^*/U$ and $V^*/V$ are isomorphic as
quadratic modules. Here, for an isotropic subgroup $U$ of a quadratic
module $M$, we use $U^*/U$ for the quadratic module with underlying group
$U^*/U$ (as quotient of abelian groups) and quadratic form $x+U\mapsto
Q_M(x)$ (note that $Q_M(x)$, for $x\in U^*$ does depend on $x$ only
modulo $U$). Note that a quadratic module $M$ is Witt equivalent to
the trivial module $0$ if and only if it contains an isotropic
self-dual subgroup.  Is is not hard to see that Witt equivalence defines
indeed an equivalence relation, and that the orthogonal sum $\perp$
induces the structure of an abelian group on the set of Witt
equivalence classes.  One can prove then~\cite{Sko 07}
(but this can also be read of from~\cite[Ch.~5,, \S1,2]{Sch 84}):
\begin{Theorem}
Two quadratic modules $M$ and $N$ are Witt equivalent if and only if
their orders are equal up to a rational square and
$\sigma_p(M)=\sigma_p(N)$ for all primes $p$.
\end{Theorem}

That Witt equivalent modules have the same order in $\Q^*/{\Q^*}^2$ and
the same $\sigma_p$-invariants is obvious (for proving equality of the
sigma invariants, say, for modules of prime power order, split, for an
isotropic submodule of $M$, the sum in the definition of $\sigma(M)$
into a double sum over a complete set of representatives $y$ for $M/U$
and a sum over $x$ in $U$). The converse statement is not needed in
this article and for its proof we refer the reader to~$\cite{Sko 07}$
or~\cite{Sch 84}.

The connection between Witt equivalence and Weil representations is
given by the following functorial property of quadratic modules. If
$U$ is an isotropic subgroup of $M$ then $U^*/U$ is again
non-degenerate, and hence we can consider its associated Weil
representation. The map $\pt {x+U} \mapsto \sum_{y\in x+U}\pt y$
defines an $\Mp$-equivariant embedding of $W(U^*/U)$ into $W(M)$.
Again, this is an immediate consequence of the defining equations for the
action of~$(T,1)$ and $(S,w_S)$.

\section{Jacobi Forms and Vector Valued Modular Forms}
\label{sec:connection-to-vector-valued-modular-forms}
If $\Gamma$
denotes a subgroup of $\SL$ we use $\JG \Gamma$ for the {\it Jacobi group}
$\JG \Gamma=\Gamma\ltimes (\Z^n\times\Z^n)$.  Thus, $\JG \Gamma$
consists of all pairs $(A,(\lambda,\mu))$ with $A\in\Gamma$, and
$\lambda,\mu\in\Z^n$, equipped with the composition law
$$
(A,(\lambda,\mu))\cdot (A',(\lambda',\mu'))
=
(AA',(\lambda,\mu)A'+(\lambda',\mu'))
.
$$
Here and in the following elements of $\Z^n$ will be considered as
column vectors.  Moreover, we use $(\lambda,\mu)A$ for $(\lambda a +
\mu c, \lambda b + \mu d)$ if $A=[a,b;c,d]$.

We identify $\Gamma$ with the subgroup $\Gamma\ltimes(0\times 0)$, and
for $\lambda,\mu\in\Z^n$ we use $[\lambda,\mu]$ for the element
$(1,(\lambda,\mu))$ of $\JG \Gamma$. Then any element $g\in\JG \Gamma$
can be written uniquely as $g=A[\lambda,\mu]$ with suitable
$A\in\Gamma$ and $\lambda,\mu$ in $\Z^n$.

Let $F$ be a symmetric, half-integral  $n\times n$~matrix.
For every integer $k$, we have an action of the Jacobi group $\JG \Gamma$
on functions
$\phi$ defined on $\HP\times\C^n$ which is given by the formulas:
\begin{gather*}
	 \phi|_{k,F}A(\tau,z)
	 =
	 \phi\left(A\tau,\frac z{c\tau+d}\right)\,
	 (c\tau+d)^{-k}\,
	 \e\left(\frac {-c F[z]}{c\tau+d}\right),\\
	 \phi|_{k,F}[\lambda,\mu](\tau,z)
	 =
	 \phi(\tau,z+\lambda\tau +\mu)\,
	 \e\left(\tau F[\lambda] + 2z^tF\lambda\right)
\end{gather*}
where $A=[a,b;c,d]\in \Gamma$ and $\lambda,\mu\in\Z^n$.

For the following its is convenient to admit also half-integral
$k$. To this end we consider the group $\JG \Gamma=\Gamma\ltimes
(\Z^n\times\Z^n)$ for subgroups $\Gamma$ of $\Mp$, which is defined as
in the case of a subgroup of $\SL$  with respect to the action
$((A,w),x)\mapsto xA$ of $\Mp$ on $\Z^2$. For half-integral
$k$, we then define $\phi|_{k,F}(A,w)$ as in the formulas above but
with the factor $(c\tau+d)^{-k}$ replaced by $w(\tau)^{-2k}$.  In this
way the symbol $|_{k,F}$ defines a right action of $\JG \Gamma$ on
functions defined on $\HP\times\C$. If $k$ is integral this action
factors through the action of $\JG {\down\Gamma}$ defined above, where
$\down\Gamma$ denotes the projection of $\Gamma$ onto its first coordinates..

\begin{Definition}
Let $F$ be a symmetric, half-integral positive definite $n\times
n$~matrix, let $k$ be a half-integral integer and, for a subgroup
$\Gamma$ of finite index in $\Mp$, let $V$ be a complex finite
dimensional $\Gamma$-module.  A {\sl Jacobi form of weight~$k$ and
index $F$ with typus $(\Gamma,V)$} is a holomorphic
function~$\phi:\HP\times \C^n\rightarrow V$ such that the following
two conditions hold true:
\begin{enumerate}
\item[(i)]
For all $\JG\Gamma$ and all $\tau\in\HP$, $z\in\C^n$, one
has $\left(\phi|_{k,F}g\right)(\tau,z) = g\left(\phi(\tau,z)\right)$,
where we view $V$ as a $\JG\Gamma$-module by letting act $\Z^n\times\Z^n$
trivially on~$V$.
\item[(ii)]
For all $\alpha\in\Mp$ the function $\phi|_{k,F}\alpha$
possesses a Fourier expansion of the form
$$
\phi|_{k,F}\alpha = \sum_{\begin{subarray}{c}l\in\Q,\; r\in\Z^{n}\\
	 4l-F^{-1}[r]\ge 0\end{subarray}} c(l,r)\,q^l\,\e(z^tr).
$$
\end{enumerate}
\end{Definition}

We shall use $J_{k,F}(\Gamma,V)$ for the complex vector space of
Jacobi forms of weight~$k$ and index $F$ of typus $(\Gamma,V)$\footnote{%
It is sometimes useful to consider more general types of Jacobi
forms. In particular, the definition does not include the case of
Jacobi forms of half-integral scalar index, the basic example for such
type being $\vartheta(\tau,z)$ (cf.~section~\ref{sec:Notations}). However, $\vartheta(\tau,2z)$ defines
an element of $J_{\frac12,2}(\Mp,\C(\varepsilon^3))$ and thus our omission merely
amounts to ignoring a certain additional invariance with respect to
the bigger lattice $\frac12\Z\times\frac12\Z$.
}.
If, for a Jacobi
form $\phi$, in condition~(i) of the definition, for all $\alpha$, the
stronger inequality $4l-F^{-1}[r] > 0$ holds true then we call $\phi$
a cusp form. The subspace of cusp forms in $J_{k,F}(\Gamma,V)$ will be
denoted by $J_{k,F}^{\text cusp}(\Gamma,V)$.

If $V$ is a $\Gamma$-module for a subgroup $\Gamma$ of $\SL$,  then we
may turn $V$ into a $\up\Gamma$-module $\up V$ by setting
$(A,w)v:=Av$, and we simply write $J_{k,F}(\Gamma,V)$ for the space
$J_{k,F}(\up \Gamma,\up V)$.  Note that in this case
$J_{k,F}(\Gamma,V)=0$ unless $k$ is integral (since then the element
$(1,-1)$ of $\up\Gamma$ acts trivially on $\up V$ and it acts as
multiplication by $(-1)^{2k}$ on $J_{k,F}(\Gamma,V)$ by the very
definition of the operator~$|_{k,F}$). Thus, if $V$ is
a $\Gamma$-module for a subgroup $\Gamma$ of $\SL$, we may
confine our considerations to integral $k$ and then the first
condition in the definition of $J_{k,F}(\Gamma,V)$ is equivalent to
the statement that $\phi|_{k,F}A(\tau,z)=A\big(\phi(\tau,z)\big)$ for
all $A\in\Gamma$.  Similarly, if $V$ is an irreducible $\Gamma$-module
for a subgroup $\Gamma$ of $\Mp$ then, for integral $k$, we have
$J_{k,F}(\Gamma,V)=0$ unless the action of $\Gamma$ on $V$ factors
through an action of $\down \Gamma$ (since, if the action of $\Gamma$
does not factor then $\Gamma$ contains $(1,-1)$, which must act
nontrivially on $V$, and, since $V$ is irreducible, the central
element $(1,-1)$ acts then as multiplication by $-1$).  Finally,
we may in principle always confine to irreducible $V$.
Indeed, if $V=\bigoplus V_j$
is the decomposition of the $\Gamma$-module $V$ into irreducible
parts, then $J_{k,F}(\Gamma,V)\cong\bigoplus J_{k,F}(\Gamma,V_j)$.

If $\C(1)$ denotes the trivial $\Gamma$-module for a subgroup
$\Gamma$ of $\SL$ or $\Mp$ we simply write~$J_{k,F}(\Gamma)$ for
$J_{k,F}\big(\Gamma,\C(1)\big)$, and we write~$J_{k,F}$
for~$J_{k,F}(\Gamma)$ if $\Gamma$ is the full
modular group.  Note that for a positive integer~$m$ and integral $k$,
the space $J_{k,m}(\Gamma)$ coincides with the usual space of Jacobi
forms of weight~$k$ and index~$m$ on $\Gamma$ as defined in~\cite{E-Z 85}.

It is an almost trivial but useful observation that, in the theory of
vector valued Jacobi or modular forms, one can always restrict
to forms on the full group~$\Mp$. In fact, one has
\begin{Lemma}
Let $\Gamma$ be a subgroup of $\Mp$ and let $V$ be a $\Gamma$-module.
Then there is a natural isomorphism
$$
J_{k,F}(\Gamma,V)\xrightarrow{\sim}J_{k,F}(\Mp,\Ind_\Gamma^{\Mp}V)
.$$
(Here $\Ind_\Gamma^{\Mp}V = \C[\Mp]\otimes_{\C[\Gamma]} V$ denotes
the $\Mp$-module induced by $V$.)
\end{Lemma}
\begin{proof}
Consider the natural map
$
\phi \mapsto \sum_{g\Gamma\in\Mp/\Gamma} g \otimes \phi|_{k,F}{g^{-1}}
$.
It is easily verified that
the summands do not depend on the choice of the representatives $g$ and that this map
defines an isomorphism as claimed in the theorem.
\end{proof}

Jacobi forms may be viewed as vector valued elliptic modular
forms. For stating this more precisely we denote by $M_k$ the space of
holomorphic functions $h$ on $\HP$ such that $h$ is a modular form of
weight $k$ on some subgroup $\Gamma$ of $\Mp$ (where we assume that
$\Gamma$ is contained in $\Gamma_0(4)$ if $k$ is not integral). Note
that $M_k$ is a $\Mp$-module with respect to the action $|_k$. By
$M_k^{\text cusp}$ we denote the submodule of cusp forms. Using the
$\Mp$-module $W(F)$ introduced in
section~\ref{sec:weil-representations} we then have

\begin{Theorem}
\label{thm:first-isomorphism}
For a subgroup $\Gamma$ of $\Mp$, let $V$ be a finite dimensional
$\Gamma$-module. Assume that the image of $\Gamma$ under the
associated representation is finite.  Then, for any half-integral $k$
and any half-integral $F$  there is a natural isomorphism
$$
J_{k,F}(\Gamma,V)
\xrightarrow{\sim}
\left(M_{k-\frac n2}\otimes W(F)^*\right)\otimes_{\C[\Mp]}\Ind_\Gamma^{\Mp}V
.
$$
\end{Theorem}
\begin{proof}
Using the isomorphism of the preceding lemma we can assume that $V$ is a $\Mp$-module.

Denote by ${\cal J}_{k,F}$ the space of all functions $\phi$, where $\phi$ is in
$J_{k,F}(\Gamma)$ for some subgroup $\Gamma$. The space ${\cal J}_{k,F}$ is clearly
a $\Mp$-module under the action~$|_{k,F}$. The bilinear map $(\phi,v)\mapsto \phi\cdot v$
induces an isomorphism
$$
{\cal J}_{k,F}\otimes_{\C[\Mp]} V \xrightarrow{\sim} J_{k,F}(\Mp,V)
.
$$
For verifying the surjectivity of this map we need that some
subgroup $\Gamma$ of finite index in $\Mp$ acts trivially on $V$ (or
equivalently that the image of $\Mp$ under the representation afforded
by the $\Mp$-module $V$ is finite).  Namely, if we write a $\phi\in
J_{k,F}(\Gamma,V)$ with respect to some basis $e_j$ of $V$ as
$\phi=\sum_j \phi_j\cdot e_j$, then, by our assumption, the $\phi_j$
are scalar valued Jacobi forms on $J_{k,F}\big(\Gamma\big)$.

Assume now that $\psi$ is an element of $J_{k,F}\big(\Gamma)$. Then it
is well-know (and follows easily from the invariance of $\psi$ under
$\Z^n\times\Z^n$) that we can expand $\psi$ in the form
$$
\psi(\tau,z)=\sum_{x\in\Z^n/2F\Z^n} h_x(\tau)\,\vartheta_{F,x}(\tau,z)
$$
with functions $h_x$ which are holomorphic in $\HP$, and where the
$\vartheta_{F,x}$ denote the theta series defined in
section~\ref{sec:Notations}. The space of functions $\Theta(F)$
spanned by the $\vartheta_{F,x}$ is a $\Mp$-(right)module with respect
to the action $|_{\frac n2,F}$ and is acted on trivially by
$\Gamma(4N)^*$ for some integer $N$ (see e.g.~\cite{Kl 46}).
Accordingly, the $h_x$ are then invariant under
$\Gamma(4N)^*\cap\Gamma$ under the action $|_{k-\frac n2}$ (one needs
here also the linear independence of the $\vartheta_{F,x}(\tau,z)$, where
$x$ runs through a set of representatives for $\Z^n/2F\Z^n$). In fact,
one easily deduces from the regularity condition for Jacobi forms at
the cusps (i.e.~from condition (ii) of the definition) that the $h_x$
are elements of $M_{k-\frac n2}$. Thus, the bilinear map
$(h,\theta)\mapsto h\cdot\theta$ induces an isomorphism of $\Mp$-right
modules
$$
M_{k-\frac n2}\otimes\Theta(F) \xrightarrow{\sim} {\cal J}_{k,F} .
$$

Finally we note that $\Theta(F)$ is isomorphic as $\Mp$-module to $W(F)^*$ via
the map $W(F)^*\ni\lambda \mapsto \sum_{e\in D_F}\lambda(e)\,\vartheta_{F,e}$
(cf.~\cite{Kl 46}). From this the theorem is now immediate.
\end{proof}

\begin{Remark}
A review of the foregoing proof shows that in the statement of the theorem
we may replace $M_{k-\frac12}$ by $M^\text{cong}_{k-\frac n2}$, the subspace of all
$h$ in $M_{k-\frac n2}$ which are modular forms on some congrence subgroup of $\Mp$,
provided $\Gamma(4N)^*$, for some $N$, acts trivially on $V$.
This remark will become important in the next section.
\end{Remark}

We use the isomorphism of Theorem~\ref{thm:first-isomorphism} to
define the subspace $J_{k,F}^{\text Eis.}(\Gamma,V)$
of {\it Jacobi-Eisenstein series}
as the preimage
of the subspace which is obtained by replacing $M_{k-\frac12}$  on the
right hand side of this isomorphism by the space of Eisenstein series
$M_{k-\frac12}^\text{Eis.}$ in $M_{k-\frac12}$ (i.e.~by the orthogonal
complement of the cusp forms in $M_{k-\frac12}$).

Theorem~\ref{thm:first-isomorphism} implies that there are no Jacobi
forms of index $F$ and weight $k < \frac n2$ (since then $M_{k-\frac
n2}=0$). For $k=\frac n2$ we have $M_k=\C$, and hence the theorem
implies
\begin{Corollary}
There exists a natural isomorphism
$$
J_{\frac n2,F}(\Gamma,V)\cong\Hom_{\Mp}(W(F),\Ind_{\Gamma}^{\Mp} V)
.
$$
\end{Corollary}
The corollary 
reduces the study of $J_{\frac n2,F}(\Gamma,V)$ to the
problem of describing the decomposition of $W(F)$ and $\Ind_{\Gamma}^{\Mp} V$ into
irreducible parts.

For $k\ge \frac n2 +2$, the last theorem makes it possible to derive an explicit formula
for the dimension of the space of Jacobi forms of weight $k$. In fact,
the isomorphism of the theorem can be rewritten as
$$
J_{\frac n2,F} \cong M_{k-\frac n2}\otimes_{\C[\Mp]}(W(F)\otimes
\Ind_{\Gamma}^{\Mp} V)
.
$$
But here (using the obvious map $h\otimes v\mapsto h\cdot v$) the
right hand side can be identified with the space $M_{k-\frac n2}(\rho)$
of holomorphic maps $h:\HP\rightarrow W(F)\otimes \Ind V$
which satisfy the transformation law
$h(A\tau)w(t)^{-2k}=\rho(\alpha)\left(h(\tau)\right)$ for all $(A,w)$
in $\Mp$, where $\rho$ denotes the representation afforded by the
$\Mp$-module $W(F)\otimes \Ind V$, and which satisfy
the usual regularity conditions at the cusps.  The dimensions of the
spaces $M_k(\rho)$ have been computed in \cite[p.~100]{Sko 85} (see
also~\cite[Sec.~4.3]{Eh-S 95}) using the Eichler-Selberg trace
formula. In our context these formulas read as follows:
\begin{Theorem}
\label{thm:dimension-formula}
Let $F$ be a half-integral positive definite $n\times n$ matrix, let
$k\in\frac12\Z$, and, for a subgroup $\Gamma$ of $\Mp$, let $V$ be a
$\Gamma$-module such that the associated representation of $\Gamma$
has finite image in $\sym{GL}(V)$.  Then one has
\begin{equation*}
\begin{split}
	 \dim &J_{k,F}(\Gamma,V) 
	 -
	 \dim M_{2+\frac n2 - k}^{\text cusp}\otimes_{\C[\Mp]} \conj{X(i^{n-2k})}\\
	 &=
	 \frac{k - \frac n2 - 1}2\,\dim X(i^{n-2k})
	 +\frac 14
	 \sym{Re}\left(e^{\pi i (k-\frac n2 )/2}\sym{tr}((S,w_S),X(i^{n-2k}))\right)\\
	 &+\frac 2{3\sqrt 3}
	 \sym{Re}\left(e^{\pi i (2k - n + 1)/6}\sym{tr}((ST,w_{ST}),X(i^{n-2k}))\right)
	 -\sum_{j=1}^r(\lambda_j-\frac12)
	 .
\end{split}
\end{equation*}
Here $X(i^{n-2k})$denotes the $\Mp$-submodule of all $v$ in
$W(F)\otimes\Ind_\Gamma^{\Mp} V$ such that $(-1,i)v=i^{n-2k}v$;
moreover, $0\le \lambda_j < 1$ are rational numbers such that
$\prod_{j=1}^r(t-e^{2\pi i\lambda_j})\in\C[t]$ equals the
characteristic polynomial of the automorphism of $X(i^{n-2k})$ given
by $v\mapsto (T,1)v$. (Recall that $T=[1,1;0,1]$ and $S=[0,-1;1,0]$).
\end{Theorem}

Note that the traces on the space $X(i^{n-2k})$ occurring on the
right hand side of the dimension formula can be easily computed from
the explicit formulas for the action of $(S,w_S)$ and $(T,w_T)$
on $W(F)\otimes \Ind V$. Note also that the second term occurring on the left hand side
vanishes for $k\ge\frac n2+2$. In general, this second term may be
interpreted as the dimension of the space $J_{2 + n - k,F}^{\text
skew}(\Gamma,V)$ of {\it skew-holomorphic} Jacobi forms, which can be
defined by requiring that Theorem~\ref{thm:first-isomorphism} holds
true for skew-holomorphic Jacobi forms if one replaces the
$\Mp$-module $M_k$ on the right hand side by $\conj{M}_k$.

It remains to discuss the case of weights $k=\frac n2 +\frac12,\frac
n2 + 1,\frac n2 +\frac32$.  For weight $k=\frac n2 + 1$ the second
term on the left hand side refers to elliptic modular forms of weight
$1$. Accordingly, this term is in general unknown.

If $k=\frac n2 +\frac32$ then the second term on the left hand side
equals the dimension of
the space of cusp forms in $J_{\frac {n+1}2,F}^{\text skew}(\Gamma,V)$,
which can be investigated by the methods of the
next section. However, we shall not pursue this any further in this
article.  Finally, the case of critical weight $k=\frac {n+1}2$ will
be considered by different methods in the next section.

To end this section we remark that Theorem~\ref{thm:first-isomorphism}
remains true if we replace on both sides of the isomorphism the space
of Jacobi forms and modular forms by the respective subspaces of cusp
forms (as can be easily read off from the proof of the
theorem). Moreover, Theorem~\ref{thm:dimension-formula}, which gives an
explicit formula for
$\dim J_{k,F}(\Gamma,V) - \dim J_{2+n-k,m}^{\text skew, cusp}(\Gamma,V)$,
remains true if the left hand side is replaced
by
$\dim J_{k,F}^{\text cusp}(\Gamma,V) - \dim J_{2+n-k,F}^{\text skew}(\Gamma,V)$
and if we subtract from the right hand side the
dimension of the maximal subspace of $X(i^{n-2k})$ which is invariant
under $(T,1)$.  For deducing this from
Theorem~\ref{thm:first-isomorphism} we refer the reader again
to~\cite[p.~100]{Sko 85}).

\section{Jacobi Forms of Critical Weight}
\label{sec:critical-weight}

In this section we study the spaces $J_{\frac {n+1}2,F}(\Gamma,V)$,
where $n$ denotes the size of~$F$.  We are mainly interested in the
cases $J_{\frac {n+1}2, F}$ and $J_{1,N}(\varepsilon^a)$. Here
$\varepsilon$ denotes the character of $\Mp=\Mpp$ given by
$\varepsilon(A,w)=\eta(A\tau)/\big(w(\tau)\eta(\tau)\big)$, where
$\eta(\tau)$ is the Dedekind eta-function. It is a well-known fact
that $\varepsilon$ generates the group of one dimensional characters
of $\Mp$.  At the end of this section we add a remark concerning the
case where~$\Gamma=\Gamma_0(l)$ and where $V$ is the trivial
$\Gamma$-module $\C(1)$; a more thorough discussion of this case with
complete proofs will be given in~\cite{I-S 07}.

We assume that $\Gamma$ is a congruence subgroup.  By
Theorem~\ref{thm:first-isomorphism} and the subsequent remark we need
first of all a description of the
$\Gamma$-module~$M_{\frac12}^{\text{cong}}$ of all modular forms of
weight $\frac12$ on congruence subgroups of $\Gamma$.  Starting with
the observation that $M_{\frac12}$ is generated by theta series
(cf.~\cite{Se-S 77}) the decomposition of the $\Mp$-module
$M_{\frac12}$ into irreducible parts was calculated
in~\cite[p.~101]{Sko 85}.  As an immediate consequence of the result
loc.~cit. we obtain
\begin{gather*}
	 M_{\frac 12}(\Gamma(4m))
	 \cong
	 \bigoplus_{\begin{subarray}{c}l|m\\
		  m/l \text{ squarefree}\end{subarray}}
	  \big(W(l)^\epsilon\big)^*
	 ,
	 \\
	 M_{\frac 12}^{\text Eis.}(\Gamma(4m))
	 \cong
	 \bigoplus_{\begin{subarray}{c}l|m\\
		  m/l \text{ squarefree}\end{subarray}}
	  \big(W(l)^{O(l)}\big)^*
\end{gather*}
Here $\epsilon$ denotes the element $\epsilon:x\mapsto -x$ of $O(l)$,
and $W(l)^\epsilon$ and $W(l)^{O(l)}$ denote the subspaces of
elements of $W(l)$ which are invariant under $\epsilon$ and
$O(l)$, respectively. (For deducing this from \cite[p.~101,
Theorem~5.2]{Sko 85} one also needs~\cite[Theorem~1.8, p.~22]{Sko 85}).

If we insert this into the isomorphism of
Theorem~\ref{thm:first-isomorphism},
we obtain
\begin{Theorem}
\label{thm:second-isomorphism}
For a congruence subgroup $\Gamma$ of $\Mp$, let $V$ be a finite dimensional
$\Gamma$-module, and let $F$ be half-integral of size $n$ and level
$f$.  Assume that, for some $m$, the group $\Gamma(4m)^*$
acts trivial on~$V$ and that $f$ divides $4m$. Then there are natural
isomorphisms
\begin{gather*}
	 J_{\frac {n+1}2,F}(\Gamma,V)
	 \longrightarrow
	 \bigoplus_{\begin{subarray}{c}l|m\\
		  m/l \text{ squarefree}\end{subarray}}
	 \Big(W\big((l)\oplus F\big)^{\epsilon\times 1}\Big)^*\otimes_{\C[\Mp]}\Ind_\Gamma^{\Mp}V
	 ,
	 \\
	 J_{\frac {n+1}2,F}^{\text Eis.}(\Gamma,V)
	 \longrightarrow
	 \bigoplus_{\begin{subarray}{c}l|m\\
		  m/l \text{ squarefree}\end{subarray}}
	 \Big(W\big((l)\oplus F\big)^{O(l)\times 1}\Big)^*\otimes_{\C[\Mp]}\Ind_\Gamma^{\Mp}V
	 .
\end{gather*}
Here $O(l)\times 1$ denotes the subgroup
of $O\big((l)\oplus F\big)$
of all elements of the form
$(x,y)\mapsto (\alpha(x),y)$ ($x\in D(l)$, $y\in D(F)$, $\alpha\in O(l)$),
and $\epsilon\times 1$ denotes the special element
of $O(l)\times 1$ given by $(x,y)\mapsto (-x,y)$.
\end{Theorem}

Thus the description of Jacobi forms of critical weight reduces to a
problem in the theory of finite dimensional representations of $\Mp$.
Actually, the description of Jacobi forms of critical weight reduces
to an even more specific problem, namely to the problem of determining
the invariants of Weil representations associated to finite quadratic
modules. To make this more precise we rewrite the first ismorphism of
Theorem~\ref{thm:second-isomorphism} as
\begin{Theorem}
\label{thm:third-isomorphism}
Under the same assumptions as in Theorem~\ref{thm:second-isomorphism}
the applications 
$$
\sum_{j}\pt{x_j}\otimes\pt{y_j}\otimes w_j
\mapsto
\sum_j\vartheta_{l,x_j}(\tau,0)\,\vartheta_{F,y_j}(\tau,z)\,w_j
$$
define an isomorphism
$$
J_{\frac {n+1}2,F}(\Gamma,V)
\longleftarrow
\bigoplus_{\begin{subarray}{c}l|m\\
	 m/l \text{ squarefree}\end{subarray}}
\Inv\Big(W(-l)^\varepsilon\otimes W(-F)\otimes \Ind_\Gamma^{\Mp}V\Big)
.
$$
\end{Theorem}
Here we used that, for any group $G$, any $G$-right module $A$ and
$G$-left module $B$, the spaces $A\otimes_{\C[G]} B$ and $(A'\otimes
B)^G$  are naturally isomorphic (where $A'$ denotes the $G$-left
module with underlying space $A$ and action $(g,a)\mapsto a\cdot
g^{-1}$), we used the isomorphism (of vector spaces) of $W(G)^c$ with
$W(-G)$ (cf.~section~\ref{sec:weil-representations}), and we wrote out
explicitly the isomorphism constructed in the proof of
Theorem~\ref{thm:first-isomorphism}.
\begin{Remark}
Next, consider a decomposition
$\Ind_\Gamma^{\Mp}V=\bigoplus_j W_j$ into irreducible $\Mp$-submodules
$W_j$.  By the results in \cite{N-W 76} every irreducible
representation of $\Mp$ is equivalent to a subrepresentation of a Weil
representation $W(M)$ for a suitable finite quadratic module
$M$\footnote{%
Actually it is proven in \cite{N-W 76} that, for prime
powers $m$, every finite dimensional irreducible
$\SL[\Ring{m}]$-module is contained in a Weil representation
associated to some finite quadratic module.  Since $\SL[\Ring{m}]$,
for arbitrary $m$, is the direct product of the groups
$\SL[\Ring{p^n}]$, where $p^n$ runs through the exact prime powers
dividing $m$, and since the category of Weil representations $W(M)$ is
closed under tensor products we can dispense with the assumption of $m$
being a prime power. Hence every irreducible $\SL$-module which is
acted on trivially by some congruence subgroup is isomorphic to a
submodule of some $W(M)$. Finally, if $\rho$ is an irreducible
representation of $\Mp$ which does not factor through a representation
of $\SL$, i.e.~which satisfies $\rho(1,-1)\not=1$, but which factors
through some congruence subgroup $\Gamma(4N)^*$ then
$\rho/\varepsilon$ factors through a representation of
$\SL[\Ring{4N'}]$ for some $N'$. Hence by the preceding argument,
$\rho/\varepsilon$ is afforded by a submodule of a Weil representation
$W(M)$. But $\varepsilon$ is afforded by a submodule in $D(6)$
(cf.~section~\ref{sec:proofs}), and hence
$\rho$ is afforded by a submodule of the Weil representation
$D(6)\otimes W(M)$.
}.
Hence we may replace the $W_j$ by submodules of
Weil representations $W(M_j)$, and at the end we find a natural
injection
$$
J_{\frac {n+1}2,F}(\Gamma,V)
\hookrightarrow
\bigoplus_j
\bigoplus_{\begin{subarray}{c}l|m\\
	 m/l \text{ squarefree}\end{subarray}}
\Inv\big(D_{-l}\perp D_{-F}\perp M_j\big)
.
$$
Here the precise image can also be characterized by the action of
the groups $O(-l)$ on $D_{-l}$ and $O(M_j)$ on $W(M_j)$ (and certain
additional intertwiners of the $\Mp$-action), so that the last
isomorphism could be written in an even more explicit form. We shall
not pursue this any further in this article.
\end{Remark}

We note a special case of
Theorem~\ref{thm:second-isomorphism}. Namely, if $V$ is the trivial
$\Mp$-module then the right hand side of, say, the first formula of
Theorem~\ref{thm:second-isomorphism} becomes the space of elements
in $W\big((l)\oplus F\big)^*$ which are invariant under $\Mp$ and
$\epsilon\times 1$.  Using again the natural isomorphism between the
spaces $W\big((l)\oplus F\big)^c$ and $W\big((-l)\oplus (-F)\big)$
(cf.~sect.~\ref{sec:weil-representations}) we thus obtain
\begin{Theorem}
\label{thm:fourth-isomorphism}
For any half-integral $F$ of size $n$ and level dividing $4m$ there
are natural isomorphisms
\begin{gather*}
	 J_{\frac {n+1}2,F}
	 \cong
	 \bigoplus_{\begin{subarray}{c}l|m\\
		  m/l \text{ squarefree}\end{subarray}}
	 \Inv\big((-l)\oplus (-F)\big)^{\epsilon\times 1}
	 ,\\
	 J_{\frac {n+1}2,F}^{\text Eis.}
	 \cong
	 \bigoplus_{\begin{subarray}{c}l|m\\
		  m/l \text{ squarefree}\end{subarray}}
	 \Inv\big((-l)\oplus (-F)\big)^{O(-l)\times 1}
	 .
\end{gather*}
(Here $O(-l)\times 1$ and $\epsilon\times 1$
are as explained in Theorem~\ref{thm:second-isomorphism}.)
\end{Theorem}

Note that $O(-l)\times 1$ acts on $\Inv\big((-l)\oplus(-F)\big)$
since the action of $O(-l)\times 1$ intertwines with the action
of $\Mp$ on the space~$W\big((-l)\oplus(-F)\big)$.
Note also that this Theorem is a trivial statement
if $n$ is even since then both sides of the claimed isomorphism are
$0$ for trivial reasons (consider the action of $(-1,1)$).  If the
size of $F$ is odd then the level of $F$ is divisible by $4$, and
hence we may choose $4m$ equal to the level of $F$.

If we take $F$ equal to a number $m$, then $W\big((-l)\oplus(-m)\big)$
does not contain nontrivial invariants (see the remark following
Theorem~\ref{thm:small-rank}.).  The theorem thus implies $J_{1,N}=0$,
a result which was proved in~\cite[Satz~6.1]{Sko 85}. More generally,
Theorem~\ref{thm:fourth-isomorphism} and the remark following
Theorem~\ref{thm:small-rank} imply
\begin{Theorem}
\label{thm:first-vanishing-result}
Let $F$ be half-integral of size $n$. If $n\not\equiv 7\bmod 8$
and $F$ has at most one nontrivial elementary divisor
then~$J_{\frac {n+1}2,F}=0$.
\end{Theorem}

Note that we cannot dispense with the assumption $n\not\equiv7\bmod8$.
A counterexample can be easily constructed as follows.
Let $2G$ denote a Gram matrix
of the $E_8$-lattice then
$$
\theta(\tau,z):=\sum_{x\in\Z^8}\e(\tau G[x] + 2z^tGx)
\quad(z\in\C^8)
$$
defines an element of $J_{4,G}$.  If $M$ is an integral
$8\times7$-matrix then $F:=M^tGM$ is half-integral positive definite
of size $7$ and it is easily checked that
$$
\theta|U_M(\tau,w):=\theta(\tau,Mw)
\qquad(w\in\C^7)
$$
defines a nonzero element of $J_{4,F}$, hence a Jacobi form of critical weight.
Suitable choices of $G$ and $M$ yield an $F$ with exactly one elementary divisor\footnote{%
One can take
$$
	 2G=
	 \begin{pmatrix}
	 4&-2&0&0&0&0&0&1\\
	 -2&2&-1&0&0&0&0&0\\
	 0&-1&2&-1&0&0&0&0\\
	 0&0&-1&2&-1&0&0&0\\
	 0&0&0&-1&2&-1&0&0\\
	 0&0&0&0&-1&2&-1&0\\
	 0&0&0&0&0&-1&2&0\\
	 1&0&0&0&0&0&0&2
	 \end{pmatrix}
$$
and $2F$ equal to the matrix which is obtained by deleting
the last row and column of $2G$, which has 4 as sole nontrivial elementary divisor.
},
i.e.~an~$F$ satisfying the second assumption of the theorem.

For doing explicit calculations it is worthwhile
to write out explicitly the isomorphisms of
Theorem~\ref{thm:fourth-isomorphism}.
If~$v=\sum_{x,y}\lambda(x,y)\,\pt{(x,y)}$ is an element
of~$\Inv\big((-l)\oplus(-F)\big)$ then
$$
\phi_v(\tau,z)
:=
\sum_{\begin{subarray}{c} x\in\Ring{2l}\\
	 y \in
	 \Z^n/2F\Z^n\end{subarray}}
\lambda(x,y)\,\vartheta_{l,x}(\tau,0)\,\vartheta_{F,y}(\tau,z)
$$
defines a Jacobi form in $J_{\frac{n+1}2,F}$. It vanishes unless
$\lambda(x,y)$ is even in $x$, i.e.~unless $v$ is invariant under
$\epsilon\times 1$, and it defines an Eisenstein series if
$\lambda(x,y)$ is invariant under $O(-l)\times 1$.

We now turn to the case $J_{1,N}(\epsilon^a)$. Using the Jacobi forms
$$
\vartheta(\tau,az)
=
q^{1/8}\big(\zeta^{a/2}-\zeta^{-a/2}\big)
\prod_{n\ge 1}
\big(1-q^n\big)
\big(1-q^n\zeta^a\big)
\big(1-q^n\zeta^{-a}\big)
$$
for positive natural numbers $a$ and trying to
build {\it thetablocks}, i.e.~Jacobi forms which are products or quotients of these forms and powers
of the Dedekind eta-function, it turns out \cite{GSZ 07} that there are
indices $N$ and nontrivial Jacobi forms in $J_{1,N}(\epsilon^a)$ if
$a\equiv 2,4,6,8,10,14\bmod 24$.  Moreover, extensive computer aided
search suggests that for all other $a$ there are indeed no Jacobi
forms, and that all Jacobi forms of weight 1 can be obtained by the
indicated procedure built on the $\vartheta(\tau,az)$.  Note that in
fact $J_{1,N}=0$ for all $N$ \cite[[Satz~6.1]{Sko 85}.  As an
application of the theory developed so far we shall prove in
section~\ref{sec:proofs}
\begin{Theorem}
\label{thm:weight-one-vanishing}
For all positive integers~$N$,
one has $J_{1,N}(\epsilon^{16})=0$.
\end{Theorem}

Actually, we shall prove more. Namely, 
\begin{Theorem}
\label{thm:quarks}
For every positive integers~$N$,
the space $J_{1,N}(\epsilon^8)$
is spanned by the series
$$
\vartheta_\rho
:=
\sum_{\alpha \in {\mathcal O}}
\left(\frac{x(\alpha)}3\right)\,q^{|\alpha|^2/3}\zeta^{y(\alpha\rho)}
$$
Here ${\mathcal O}=\Z\big[\frac{-1+\sqrt{-3}}2\big]$, and we use
$x(\alpha)=\alpha+\overline\alpha$ and
$y(\alpha)=(\alpha-\overline\alpha)/\sqrt{-3}$. Moreover, $\rho$ runs
through all numbers in ${\mathcal O}$ with $|\rho|^2=N$.
\end{Theorem}

\begin{Remark}
Let $\rho=\frac{p+\sqrt{-3}q}2$ be a number in ${\mathcal O}$ with
$|\rho|^2=N$. By multiplying~$\rho$ by a suitable 6th root of unity
(which will not change $\vartheta_\rho$) we may assume that
$q\ge|p|>0$. For $q=|p|$ one has $\vartheta_\rho=0$. For $q>|p|$, one
can show by elementary transformations of the involved series
\cite{GSZ 07} that $\vartheta_\rho$ has the factorization
$$
\vartheta_\rho(\tau,z)
=-
\vartheta\big(\tau,\frac{q+p}2 z\big)
\vartheta\big(\tau,\frac{q-p}2 z\big)
\vartheta\big(\tau,q z\big)
/\eta(\tau).
$$
\end{Remark}
Similar theorems can be produced for the spaces $J_{1,N}(\varepsilon^a)$
for arbitrary even integers $a$ modulo 24. Since the analysis of the
invariants in the corresponding Weil representations becomes quite
involved this will eventually be treated elsewhere. 

Finally, we mention the case $\Gamma=\Gamma_0(l)$ acting trivially on
$V=\C(1)$. On investigating the representation $\Ind V$ occurring on
the right hand side of Theorem~\ref{thm:second-isomorphism}
it is possible to deduce the following~\cite{I-S 07}
\begin{Theorem}
\label{thm:vanishing-on-Gamma_0}
Let $F$ be a symmetric, half-integral and positive definite matrix of
size $n\times n$ with odd $n$.  Then, for every positive integer $l$,
we have
$$
J_{\frac {n+1}2,F}\big(\Gamma_0(l)\big)
=
J_{\frac {n+1}2,F}\big(\Gamma_0(l_1)\big).
$$
Here $l_1$ is the first factor in the decomposition $l=l_1l_2$,
where $l_1 |\det(2F)^\infty$ and where $l_2$ is relatively prime to
$\det(2F)$.
\end{Theorem}

For $n=1$, i.e.~for the case of ordinary Jacobi forms in one variable,
the theorem was already stated and proved in \cite{I-S 06}.

As immediate consequence we obtain
\begin{Corollary}
Suppose $J_{\frac {n+1}2,F}=0$. Then we
have~$J_{\frac {n+1}2,F}(\Gamma_0(l))=0$
for all $l$ which are relatively prime to
$\det(2F)$.
\end{Corollary}
This corollary might be meaningful for the study of Siegel modular forms of critical
weight on subgroups $\Gamma_9(l)$. In \cite{I-S 06} we used it to prove that
there are no Siegel cusp forms of weigt one on $\Gamma_0(l)$, for any $l$.

\section{Proofs}
\label{sec:proofs}

In this section we append the proofs of
Theorems~\ref{thm:hauptvermutung-for-small-rank},~\ref{thm:weight-one-vanishing}
and~\ref{thm:quarks}.

We start with the description of the decomposition of Weil
representations of rank 1 modules into irreducible parts. For an odd
prime power $q=p^\alpha\ge1$ let $L_q$ be the quadratic module
$\big(\Ring{q},\frac{x^2}q\big)$. The submodule
$U:=p^{\lceil\alpha/2\rceil}L_q$ is isotropic, for its dual one finds
$U^*=p^{\lfloor\alpha/2\rfloor}L_q$ and, for $\alpha\ge 2$, the
quotient module $U^*/U$ is isomorphic (as quadratic module) to
$L_{q/p^2}$.  Hence $W(L_{q/p^2})$ embeds naturally as $\Mp$-submodule
into $W(L_q)$ (see the discussion at the end of
section~\ref{sec:weil-representations}).  Let $W_1(L_q)$ be the
orthogonal complement of $W(L_{q/p^2})$ (with respect to the $\Mp$-invariant
scalar product $(-,-)$ given by $(\pt x,\pt y)=1$ if $x=y$ and $(\pt x,\pt
y)=0$ otherwise).  Then $W_1(L_q)$ is invariant under $\Mp$ and under
$O(L_q)$. The latter group is generated by the involution
$\alpha:x\mapsto -x$. For $\epsilon=\pm1$, let $W_1^\epsilon(L_q)$ be
the $\epsilon$-eigenspace of $\alpha$ viewed as involution on
$W_1(L_q)$. Since $\alpha$ intertwines with $\Mp$ these eigenspaces
are $\Mp$-submodules of $W(L_q)$. By induction we thus obtain the
decomposition
$$
W(L_q)
=
\bigoplus_{d^2|q}
\left(W_1^+(L_{q/d^2})\oplus W_1^-(L_{q/d^2})\right)
\oplus
\begin{cases}
  W(L_1)&\text{if $q$ is a square}\\
  0&\text{otherwise}
\end{cases}
$$
of $W(L_q)$ into $\Mp$-submodules. Note that $W(L_1)\cong\C(1)$.

Finally, for an integer $a$ which is relatively prime to $q$ we let
$L_q(a)$ be the quadratic module which has the same underlying abelian
group as $L_q$ but has quadratic form
$Q_{L_q(a)}(x)=aQ_{L_q}(x)$. With $W_1^\epsilon\big(L_q(a)\big)$ being
defined similarly as before we obtain a corresponding decomposition
of $W\big(L_q(a)\big)$ as for $W(L_q)$.

Note that $\sigma\big(L_q(a)\big)=1$ if $q$ is a square, and otherwise
$\sigma\big(L_q(a)\big)=\left(\frac {-a}p\right)$ if $p\equiv 1\bmod
4$ and $\sigma\big(L_q(a)\big)=\left(\frac {-a}p\right)i$ if $p\equiv
3\bmod 4$. We may restate this as $\sigma(\big(L_q(a)\big)=\left(\frac
{-a}p\right)^\alpha\sqrt{\left(\frac {-4}p\right)^\alpha}$ for
$q=p^\alpha$.  In particular, $W\big(L_q(a)\big)$ can be viewed as
$\SL$-module. The level of $L_q(a)$ equals $q$, whence $\Gamma(q)$
acts trivially on $W\big(L_q(a)\big)$.

\begin{Lemma}
\label{lem:Lqa-properties}
The $\SL$-modules $W_1^\epsilon\big(L_q(a)\big)$ are irreducible. The
exact level of $W_1^\epsilon\big(L_q(a)\big)$ equals $q$ (i.e.~$q$ is
the smallest positive integer such that $\Gamma(q)$ acts trivially on
$W_1^\epsilon\big(L_q(a)\big)$). Two $\Mp$-modules
$W_1^{\epsilon}\big(L_q(a)\big)$ and
$W_1^{\epsilon'}\big(L_{q'}(a')\big)$ are isomorphic if and only if
$q=q'$, $\epsilon=\epsilon'$ and $aa'$ is a square module $p$.
\end{Lemma}
\begin{proof}
It is easy to see that the modules $W_1^\epsilon\big(L_q(a)\big)$ are
nonzero. Thus the given decomposition of $W\big(L_q(a)\big)$ contains
$\sigma_0(q)$ non zero parts. Hence, for proving the irreducibility
of these parts is suffices to show that the number of irreducible
components of $W\big(L_q(a)\big)$ is exactly $\sigma_0(q)$. But the
number of irreducible components equals the number of invariant elements
of~$W\big(L_q(a)\big)\otimes W\big(L_q(-a)\big)$. i.e.~the dimension of
$\Inv\big(L_q(a)\perp -L_q(a)\big)$. The latter space is spanned by the elements
$I_U:=\sum_{x\in U}\pt x$, where $U$ runs through the isotropic
self-dual modules of $L_q(a)\perp -L_q(a)\big)$
(cf.~Theorem~\ref{thm:simple-invariants}).

Alternatively, for avoiding the use of
Theorem~\ref{thm:simple-invariants}, which we did not prove in this
article, (but see \cite[Theorem 5.5.7]{N-S-R 07}), one can show that
$W\big(L_q(-a)\big)=\Inv\big(L_q(a)\perp -L_q(a)\big)$ is isomorphic
to the permutation representation of $\SL$ given by the right action of
$\SL$ on the row vectors of length 2 with entries in $\Ring{q}$
(see~e.g.~\cite[\S~3]{N-W 76}). The number of invariants equals thus
the number of orbits of this action, which in turn are naturally parameterized
by the divisors of $q$ (all vectors $(x,y)$ with $\sym{gcd}(x,y,q)=d$,
for fixed $d|q$ constitute an orbit).

We already saw that $\Gamma(q)$ acts trivially on
$W_1^\epsilon\big(L_q(a)\big)$. That $q$ is the smallest integer with
this property follows from the fact that $\e(a/q)$ is an eigenvalue of
$T$ acting on $W_1^\epsilon\big(L_q(a)\big)$. The latter can directly
be deduced from the definition of $W_1^\epsilon\big(L_q(a)\big)$ (see
also~\cite[p.~33]{Sko 85} for details of this argument).

Finally, if $W_1^{\epsilon}\big(L_q(a)\big)$ and
$W_1^{\epsilon'}\big(L_{q'}(a')\big)$ are isomorphic then, by
comparing the levels of the $\SL$-modules in question, we conclude
$q=q'$. The eigenvalues of $T$ on each of these spaces are of the form
$\e_q(ax^2)$ and $\e_q(a'x^2)$, respectively, for suitable integers
$x$, with at least one $x$ relatively prime to $q$; whence $aa'$ must be
a square modulo $q$. But then  $W_1^{\epsilon'}\big(L_{q'}(a')\big)$,
(and hence, by assumption, $W_1^{\epsilon}\big(L_{q}(a)\big)$), is
isomorphic to $W_1^{\epsilon'}\big(L_{q}(a)\big)$. This finally
implies $\epsilon=\epsilon'$, since $-1=S^2$ acts by $\pt x\mapsto
\sigma^2\pt{-x}=\sigma^2(\alpha\cdot\pt x)$ on $L_q(a)$ (where
$\sigma=\sigma\big(L_q(a)\big)$) as follows from the explicit formula
for the action of $S$.
\end{proof}

Similarly, we can decompose the representations associated to the
modules $D_{2^\alpha}$ or, more generally, to the modules
$D_{2^\alpha}(a):=\big(\Ring{2q},{ax^2}/{4q}\big)$, where $a$ is an
odd number and $q=2^\alpha\ge1$. As before, $W\big(D_{q/4}(a)\big)$
embeds naturally into $W\big(D_{q}(a)\big)$ and we define
$W_1^\epsilon\big(D_{q}(a)\big)$ analogously to
$W_1^\epsilon\big(L_{q}(a)\big)$. Again, $W\big(D_{2^\alpha}(a)\big)$
decomposes as direct sum of all the modules
$W_1^\epsilon\big(D_{2^\alpha/d^2}(a)\big)$ (plus $W(D_1)$ if
$2^\alpha$ is square) with $d^2$ running through the square divisors
of $2^\alpha$ and~$\epsilon=\pm 1$. Moreover, the preceding lemma still
holds true with suitable modifications when  replacing $L_q(a)$ by
$D_{2^\alpha}(a)$. Note, however, that $D_{2^\alpha}(a)$ does not
factor through a representation of $\SL$ since
$\sigma\big(D_{2^\alpha}(a)\big)=\e_8(-a)$.\footnote{%
For $\alpha=0$,
this formula holds only true for $a=\pm1$, which, however, can always
be assumed without loss of generality.
}
\begin{Lemma}
Let $q$ be a power of $2$.
The $\SL$-modules $W_1^\epsilon\big(D_q(a)\big)$ are irreducible. The
exact level of $W_1^\epsilon\big(D_q(a)\big)$ equals $4q$ (i.e.~$4q$ is
the smallest positive integer such that $\Gamma(q)^*$ acts
trivially). Two $\Mp$-modules $W_1^{\epsilon}\big(D_q(a)\big)$ and
$W_1^{\epsilon'}\big(D_{q'}(a')\big)$ are isomorphic if and only if
$q=q'$, $\epsilon=\epsilon'$ and $aa'$ is a square modula $4q$.
\end{Lemma}
The proof is essentially the same as for the preceding lemma and we leave
it to the reader.

Note that the modules $D_m$, for arbitrary nonzero $m$, and, more
generally, the modules $D_m(a):=\big(\Ring{2m},ax^2/4m\big)$, for $a$
relatively prime to $m$, have as $p$-parts modules of the form $L_q(a')$ and
$D_{2^\alpha}(a'')$, whence $W\big(D_m(a)\big)$ is isomorphic to the
outer ternsor product of $\SL/\Gamma(q)$-modules $W\big(L_q(a')\big)$
and a suitable $\Mp/\Gamma(4\cdot 2^\alpha)^*$-module
$W\big(D_{2^\alpha}(a'')\big)$. We conclude that the $\Mp$-module
$W\big(D_m(a)\big)$ decomposes as
$$
W\big(D_m(a)\big)
\cong
\bigoplus_{\begin{subarray}{c}fd^2|m\\
f \text{squarefree}\end{subarray}}
W_1^f\big(D_{m/d^2}(a)\big)
.
$$
Here $W_1^f\big(D_{m}(a)\big)$ is the subspace of all $v=\sum_{x\in
D_m}\lambda(x)\pt x$ in $W\big(D_m(a)\big)$ such that, for all
isotropic submodules $U$ of $D_m(a)$ and all $y\in U^*$, one has
$\sum_{x\in y+U}\lambda(x)=0$, and such $gv=\chi_f(g)v$ for all $g$ in
$O(m)$. Here $\chi_f$ denotes that character of $O(m)$ which maps
$g_p$ to $-1$ if $p|f$ and to $+1$ otherwise, where $g_p$ is the
orthogonal map which corresponds to the residue class $x$ in $\Ring{2m}$
(under the correspondence described in
section~\ref{sec:weil-representations}) which satisfies $x\equiv
-1\bmod 2p^\alpha$ and $x\equiv +1\bmod 2m/p^\alpha$ with $p^\alpha$
being the exact power of $p$ dividing $m$.  Note that, for a fixed
$U$, the vanishing condition   can be restated as $v$ being orthogonal
to the image of the quadratic module $U^*/U$ under the embedding of
$W(U^*/U)$ into $W\big(D_m(a)\big)$ and with respect to the scalar
product as described above.  Note also that $U_d:=\frac{m}d D_m$ runs
through all isotropic subgroups of $D_m$ if $d$ runs through all
positive integers whose square divides $m$, and that, for such $d$,
the quadratic module $U_d^*/U_d$ is isomorphic to $D_{m/d^2}$.

The decomposition of $W(m)$ was already given (and essentially deduced by
the same methods as explained here) in~\cite[Theorem 1.8, p.~22]{Sko 85}.

After these preparations we can now give the proofs of
Theorems~\ref{thm:hauptvermutung-for-small-rank}
and~\ref{thm:weight-one-vanishing}.

\begin{proof}[Proof of Theorem~\ref{thm:hauptvermutung-for-small-rank}]
Suppose first of all that $p$ is odd.  The assumptions on~$M$ imply a
decomposition $M\cong L_{p^\alpha}(a)\perp L_{p^\beta}(b)$ with, say,
$0 \le \alpha \le \beta$ \cite[Satz~3]{N 76}.  If $|M|$ is a square
and $\sigma(M)=1$ then $\alpha$, $\beta$ are both even or they are
both odd and $\left(\frac{-ab}p\right)(=\sigma(M))=1$. In the first case
$M$ contains the trivial submodule as follows immidiately from
Lemma~\ref{lem:Lqa-properties}.  In the second case we may assume
$b=-a$; but then, as follows again from
Lemma~\ref{lem:Lqa-properties}, $A:=W\big(L_{p^\alpha}(a)\big)$ is a
$\SL$-submodule of $B:=W\big(L_{p^\beta}(-b)\big)$, whence $W(M)\cong
A\otimes B^c$ contains nonzero invariants.

Conversely, if $\alpha$ and
$\beta$ do not have the same parity, or if $-ab$ is not a square
modulo $p$, then by Lemma~\ref{lem:Lqa-properties}, $A$ and $B$ have
no irreducible representation in common, whence $W(M)\cong A\otimes
B^c$ contains no invariants.

If $p=2$ and $M\cong D_{2^\alpha}(a)\perp D_{2^\beta}(b)$ with, say,
$\alpha\le \beta$, then $\sigma(M)=\e_8(-a-b)$ (using the convention
that $a$ or $b$ is chosen to be $\pm1$ if $\alpha$ or $\beta$ are
equal to zero, respectively). Hence $\sigma(M)=1$ if and only if
$b\equiv -ax^2\bmod 4\cdot 2^\beta$ for some $x$.  We can therefore
follow the same line of reasoning as before to deduce the claim.

Finally, if $p=2$ and $M$ is not of the form as in the preceding
argument then~\cite[Satz~4]{N 76}
$$
M\cong \big((\Ring{2^\alpha})^2,\frac{xy}{2^\alpha}\big)
\quad\text{or}\quad
M\cong\big((\Ring{2^\alpha})^2,\frac{x^2+xy+y^2}{2^\alpha}\big)
.
$$
In particular, $|M|$ is a square.  In the first case $\sigma(M)=1$
and $W(M)$ possesses an invariant, namely $I_U=\sum_{x\in U}\pt x$,
where,~e.g.~$U=\Ring{2^\alpha}\times 0$.  In the second case
$\sigma(M)=(-1)^\alpha$. If $\alpha$ is even then $I_V=\sum_{x\in
V}\pt x$, for $V=2^{\alpha/2}M$, is an invariant. If $\alpha$ is odd
then  $W(M)$ possesses no invariants (see the complete decomposition
of $W(M)$ in~\cite[pp.~519]{N-W 76}, which is called $N_\alpha$
loc.~cit.).  This completes the proof of the theorem.
\end{proof}

\begin{proof}[Proof of Theorems~\ref{thm:weight-one-vanishing} and \ref{thm:quarks}]
The $\Mp$-module $L^{-}_3(s)$ ($s=\pm1$) is one-dimen\-sional, and $T$
acts on it by multiplication with $e_3(s)$. Hence the character
afforded by $L^{-}_3(s)$ equals $\epsilon^{8s}$. By
Theorem~\ref{thm:third-isomorphism} the space $J_{1,N}$ embeds
injectively into the direct sum of the spaces $\Inv(M_l)$, where
$$
M_l:=D_{-l}\perp D_{-N}\perp L_3(s)
$$
and where $l$ runs through the divisors of $N':=\sym{lcm}(3,N)$
such that $N'/l$ is squarefree.

Suppose that $\Inv(M_l)\not=0$. Since
$\Inv(M_l)=\bigotimes_k\Inv\big(M_l(p)\big)$, where, $M_l(p)$, for a
prime $p$, denotes the $p$-part of $M_l$, we conclude
$\Inv(M_l(p))\not=0$. For $p\not=3$ this implies, by
Theorem~\ref{thm:small-rank}, that the $p$-parts of $|M_l|=3Nl$ are
perfect squares and that $\sigma_p(M_l)=1$. From the first condition
and since $N'/l$ is squarefree we conclude $l=N'$ or $l=N'/3$. But
$\sigma_2(M_l)=e_8(-l/q -N/q)\left(\frac{lN/q^2}{2q}\right)$, where
$q$ is the exact power of $2$ dividing $N$, and this is real if only
if $ln/q^2$ is not a square mod 4. Thus $l=N/3$ or $l=3N$ accordingly
as $3$ divides $N$ or not. For this $l$ and $p\not=3$ we have
$\sigma_p(M_l)=\left(\frac{-3}{q_p}\right)$, where $q_p$, for any $p$,
denotes the exact power of $p$ dividing $N$. Since $\sigma_p(M)=1$ we
find that $\left(\frac{-3}{q_p}\right)=1$ for all $p\not=3$. In
particular, $N/q_3\equiv +1 \bmod 3$.  Finally, if, say, $3|N$ and
$l=3N$, we find
$$
M_l(3)=\big(\Ring{3q_3}\times\Ring{3q_3}\times\Ring{q_3},\frac{-N/q_3x^2-3N/q_3y^2+q_3
s z^2}{3q_3}\big) .
$$
If $N/q_3\equiv +1\bmod 3$ then, for $s=-1$, this quadratic module
contains no nontrivial isotropic element, hence $T$ does not afford
the eigenvalue 1 on $W\big(M_l(3)\big)$. Hence $\Inv(M_l(3))\not=0$
implies $s=+1$.  The same holds true, by a similar argument, if $N$ is
not divisible by $3$. Note that, for $s=+1$ and $N/q_3\equiv +1\bmod
3$ we have $\sigma_3(M_l)=1$.

Summing up we thus have proved that $J_{1,N}(\varepsilon^{8s})=0$
unless $s=+1$ and~$\left(\frac{-3}{q_p}\right)=1$ for all
$p\not=3$. Moreover, if the latter conditions hold true then
$$
J_{1,N}(\varepsilon^{8})\cong
\begin{cases}
  \Inv(M_{3N})^{+,-}&\text{if }3\not|N\\
  \Inv(M_{N/3})^{+,-}&\text{if }3|N
\end{cases}
.
$$
Here the superscript indicates the subspace of all elements in
$W(M_{3N})$ resp. $W(M_{N/3})$ which are even in the first and odd in
the third argument.

By Theorem~\ref{thm:simple-invariants} the space $\Inv(M_{3N})$ is
spanned by the special elements $\pt U=\sum_{m\in U}\pt m$, where $U$
runs through the isotropic self-dual subgroups of $M_{3N}$.
Accordingly, $J_{1,N}(\varepsilon^{8})$ is the spanned by the Jacobi
forms
$$
\vartheta_U(\tau,z)
:=
\sum_{(x,y,z)\in U}
\vartheta_{3N,x}(\tau,0)\,\vartheta_{N,y}(\tau,z)
\left(\frac z3\right) .
$$
Here we used the application of Theorem~\ref{thm:third-isomorphism}
(and the identification $ L_3^-\xrightarrow{\sim}\C(\varepsilon^8)$
given by $\pt z-\pt{-z} \mapsto \left(\frac z3\right)$).

The statements of the last paragraph still hold true if $N$ is
divisible by $3$. In fact, a literal application of
Theorem~\ref{thm:third-isomorphism} and the preceding considerations
imply that $J_{1,N}(\varepsilon^{8})$ is spanned by the the Jacobi
forms which are given by the same formula as the $\vartheta_U$ but
with $\vartheta_{3N,x}$ replaced by $\vartheta_{N/3,x}$ and where $U$
runs through the isotropic self-dual subgroups of $M_{N/3}$. But the
quadratic module $D_{-N/3}$ is isomorphic to $X^*/X$ (via $x\mapsto
3x+X$), where $X^*=3\Ring{6N}$. This map induces an embedding of
$\Mp$-modules $W(-N/3)\rightarrow W(-3N)$ (via $(x)\mapsto
\sum_{y\equiv 3x\bmod 2N}(y)$; see end of
section~\ref{sec:weil-representations}), and it induces a map
$U\mapsto U'$ from the set of isotropic self-dual subgroups of
$M_{n/3}$ into the set of isotropic self-dual subgroups of $M_{3N}$
(via pullback). Note that this map is one to one
(since~$-x^2-3y^2+4Nz^2 \equiv 0 \bmod 12N$ implies $3|x$). Finally, the
diagram formed by the maps $\Inv(M_{N/3})\ni U\mapsto e_U\mapsto
\vartheta_U$ and $\Inv(M_{N/3})\ni U\mapsto U'\mapsto  \vartheta_{U'}$
commute as follows on using the formulas
$$
\vartheta_{N/3,x} = \sum_{y\equiv 3x\bmod 2N} \vartheta_{3N,y} .
$$

It remains to determine, for arbitrary $N$, the isotropic self-dual
subgroups $U$ of $M_{3N}$. The map $U\ni (x,y,z)\mapsto (x,y)$ is
injective (since~$-\frac{x^2+3y^2}{4N}+z^2\in 3\Z$) and thus maps $U$
onto an isotropic subgroup $U'$ of the (degenerate) quadratic module
$M:=\big(\Ring{6N}\times\Ring{2N},-\frac{x^2+3y^2}{4N}\big)$ of order
$|U'|=6N$. For determining the set $S$ of isotropic subgroups $U'$ of
$M$ whose order is $6N$ let $\mathcal O$ be the maximal order of
$\Q\big(\sqrt{-3}\big)$. The image of the map ${\mathcal O}\ni \alpha
=\frac{x+\sqrt{-3}y}2\mapsto
\big(x+6N\Z,y+2N\Z\big)$ contains the
isotropic vectors of~$M$ (since $\frac{x^2+3y^2}{4N}\in\Z$ implies
that $x$ and $y$ have the same parity) and its kernel equals
$2\sqrt{-3}N\Z[\sqrt{-3}]$. The $U'$ in $S$ thus correspond to the
subgroups $2\sqrt{-3}N\Z[\sqrt{-3}]\subset I\subset{\mathcal O}$ of
index $N$ in ${\mathcal O}$ and such that $N||\alpha|^2$ for all
$\alpha\in I$. Let $I$ be such a subgroup and let $I'$ be the
${\mathcal O}$-ideal generated by $I$. Clearly $N||\alpha|^2$ for all
$\alpha$ in $I'$. Since $I'={\mathcal O}\rho$ for some $\rho$ we find
$N||\rho|^2$ and $|\rho|^2\le N$. But then $|\rho|^2=N$ and
$I'=I$. Thus all isotropic self-dual subgroups $U$ of $M_{3N}$ are of
the form
$$
U
=
\left\{ \big(x(\alpha)+6N\Z,y(\alpha)+2N\Z,\psi(\alpha)\big) :
\alpha\in {\mathcal O}\rho\right\},
$$
where ${\mathcal O}\rho$ is an ideal of norm $N$, where $\psi$
is a group homomorphism ${\mathcal O}\rho\rightarrow \Ring{3}$ such
that $-|\alpha|^2/N+\psi(\alpha)^2\in 3\Z$, and where we use
$x(\alpha)=\alpha+\overline\alpha$, $y(\alpha)=(\alpha-\overline\alpha)/\sqrt{-3}$.
  It is easily checked that
the only possible $\psi$ are $\psi(\alpha)=\epsilon x(\alpha/\rho)$
($\epsilon=\pm 1$).  In view of the formula for $\vartheta_U$ it
suffices to consider only those $U=U_\rho$ where
$\psi(\alpha)=x(\alpha/\rho)$. If we write $\vartheta_\rho$ for
$\vartheta_U$ then by a simple calculation
$$
\vartheta_\rho = \sum_{\alpha \in {\mathcal O}}
\left(\frac{x(\alpha)}3\right)\,q^{|\alpha|^2/3}\zeta^{y(\alpha\rho)} ,
$$
We have thus have proved that $J_{1,N}(\varepsilon^{8})$ is indeed spanned
by the series $\vartheta_\rho$ as stated in Theorem~\ref{thm:quarks},
where ${\mathcal O}\rho$ runs through
the ideals of norm $N$ in~${\mathcal O}$.
\end{proof}


\bigskip
\begin{flushleft}
Nils-Peter Skoruppa\\
Fachbereich Mathematik,
Universit\"at Siegen\\
Walter-Flex-Stra{\ss}e 3,
57068 Siegen, Germany\\
http:/\hskip-3pt/www.countnumber.de\\
\end{flushleft}

\end{document}